\documentclass[12pt,a4paper,reqno]{amsart}
\usepackage{amssymb}
\usepackage{latexsym}
\usepackage{exscale}

\numberwithin{equation}{section}
\newtheorem{theor}{Theorem}[section]

\newtheorem{defi}[theor]{Definition}
\newtheorem{corol}[theor]{Corollary}
\newtheorem{remark}[theor]{Remark}
\newtheorem{prop}[theor]{Proposition}
\newtheorem{defin}[theor]{Definition}

\newcommand{\re}{\mathbb{R}}
\newcommand{\co}{\mathbb{C}}

\newcommand{\RR}{\mathbb{R}}
\newcommand{\NN}{\mathbb{N}}

\newcommand{\J}{\mathcal{J}}
\newcommand{\K}{\mathcal{K}}

\newcommand{\W}{\mathcal{W}}

\newcommand{\D}{\mathcal{D}}

\newcommand{\cals}{\mathcal{S}}

\newcommand{\A}{\mathcal{A}}

\newcommand{\bigchi}{\mathop{\mathchoice%
{\mbox{\Large$\chi$}}{\mbox{\large$\chi$}}{\mbox{\normalsize$\chi$}}%
{\mbox{\small$\chi$}}}\nolimits}

\newcommand{\BMO}{{\rm BMO}}

\renewcommand{\emptyset}{\mbox{\rm \O}}
\renewcommand{\Re}{{\rm Re}\,}

\def\div{\mathop{\rm div}}
\def\Int{\mathop{\rm Int}}
\def\supp{\mathop{\rm supp}}

\newcommand{\aver}[1]{-\hskip-0.46cm\int_{#1}}
\newcommand{\textaver}[1]{-\hskip-0.40cm\int_{#1}}

\newcommand{\expt}[1]{e^{\textstyle #1}}

\newcommand{\off}[2]{\mathcal{O}\big(L^{#1}-L^{#2}\big)}
\newcommand{\offw}[2]{\mathcal{O}\big(L^{#1}(w)-L^{#2}(w)\big)}
\newcommand{\fullx}[2]{\mathcal{F}\big(L^{#1}-L^{#2}\big)}

\newcommand{\dec}[1]{\Upsilon\!\left(#1\right)}

\begin{document}
\allowdisplaybreaks

\title[Weighted norm inequalities and elliptic operators]
{Weighted norm inequalities, off-diagonal estimates and elliptic
operators}

\author{Pascal Auscher}

\address{Pascal Auscher
\\
Universit\'{e} de Paris-Sud et CNRS UMR 8628
\\
91405 Orsay Cedex, France} \email{pascal.auscher@math.u-psud.fr}

\author{Jos\'{e} Mar\'{\i}a Martell}

\address{Jos\'{e} Mar\'{\i}a Martell
\\
Instituto de Ciencias Matem\'{a}ticas CSIC-UAM-UC3M-UCM
\\
Consejo Superior de Investigaciones Cient\'{\i}ficas
\\
C/ Serrano 121 \\ E-28006 Madrid, Spain}

\email{chema.martell@uam.es}

\thanks{This work was partially supported by the European Union
(IHP Network ``Harmonic Analysis and Related Problems'' 2002-2006,
Contract HPRN-CT-2001-00273-HARP).
The second author was also supported by MEC ``Programa Ram\'{o}n y Cajal, 2005'', by MEC Grant MTM2007-60952, and by UAM-CM Grant CCG07-UAM/ESP-1664.}

\date{\today}


\begin{abstract}
We give an overview of the  generalized
Calder\'{o}n-Zygmund theory for ``non-integral'' singular operators, that is, operators without kernels bounds but appropriate off-diagonal estimates. This theory is powerful enough to obtain weighted
estimates for such operators and their  commutators with $\BMO$ functions.
 $L^p-L^q$ off-diagonal estimates when $p\le q$ play an important role and we present them. They are particularly well suited  to the semigroups generated by second
order elliptic operators and the range of exponents $(p,q)$ rules the $L^p$ theory for many   operators constructed from the semigroup and its gradient. Such applications are summarized.
\end{abstract}

\maketitle

\section{Introduction}

The Hilbert transform in $\re$ and the Riesz transforms in $\re^n$ are prototypes of Calder\'{o}n-Zygmund operators. They are singular integral operators represented by kernels with some decay and smoothness. Since the 50's, Calder\'{o}n-Zygmund operators have been thoroughly studied. One first shows that the operator in question is bounded on $L^{2}$  using  spectral theory, Fourier transform or even the powerful $T(1)$, $T(b)$ theorems. Then, the smoothness of the kernel and the Calder\'{o}n-Zygmund decomposition lead to the  weak-type (1,1) estimate, hence strong type $(p,p)$ for $1<p<2$. For $p>2$, one uses duality or interpolation from the $L^\infty$ to $\BMO$ estimate, which involves also the regularity of the kernel. Still another way for $p>2$ relies on good-$\lambda$
estimates via the Fefferman-Stein sharp maximal function. It is interesting to note that both Calder\'{o}n-Zygmund decomposition and good-$\lambda$ arguments use independent smoothness conditions on
the kernel, allowing generalizations in various ways. Weighted estimates for these operators can be proved by means of the Fefferman-Stein sharp maximal function, one shows boundedness on $L^p(w)$ for every $1<p<\infty$ and $w\in A_p$, and a weighted weak-type $(1,1)$ for weights in $A_1$. Again, the smoothness of the kernel plays a crucial role.  We refer the reader to \cite{Gra} and \cite{GR} for more details on this topic.

It is natural to wonder whether the smoothness of the kernel is needed or, even more, whether one can develop a generalized Calder\'{o}n-Zygmund theory in absence of kernels. Indeed, one finds
Calder\'{o}n-Zygmund like operators  without any (reasonable) information on their kernels which, following the implicit terminology
introduced in \cite{BK1}, can be called singular ``non-integral''  operators in the sense that they are still of order 0  but they do not have an
integral representation by a kernel with size and/or smoothness
estimates. The goal is to obtain some range of exponents $p$ for
which $L^p$ boundedness holds, and because this range may not be
$(1,\infty)$, one should abandon any use of kernels. Also, one seeks for weighted estimates trying to determine for which class of Muckenhoupt  these operators are bounded on $L^p(w)$. Again, because the range of the unweighted estimates  can be a proper subset of $(1,\infty)$ the class $A_p$, and even the smaller class $A_1$, might be too large.

The generalized Calder\'{o}n-Zygmund theory  allows us to reach this goal: much of all the classical results extend. As a direct application, we show in Corollary \ref{corol:CZO} that assuming that for a bounded (sub)linear operator $T$ on $L^2$, the boundedness on $L^p$ ---and even on $L^p(w)$ for $A_p$ weights---  follows
from two basic inequalities involving the operator and its action on some functions and not its kernel:
\begin{equation}\label{eq:CZO1}
\int_{\RR^n \setminus 4B} |Tf(x)|\, dx
\le C \int_B |f(x)|\, dx,
\end{equation}
for any ball $B$ and any bounded function $f$ supported on $B$ with
mean $0$, and
\begin{equation}\label{eq:CZO2}
\sup_{x\in B}   |Tf(x)|
\le
C \aver{2B} |Tf(x)|\, dx +  C\inf_{x\in B} Mf(x),
\end{equation}
for any ball  $B$ and any bounded  function $f$ supported on
$\re^n\setminus 4\,B$.  The first condition is used to go below
$p=2$, that is, to obtain that $T$ is of weak-type $(1,1)$. On the
other hand, \eqref{eq:CZO2} yields the estimates for $p>2$ and also
the weighted norm inequalities in $L^p(w)$ for $w\in A_p$,
$1<p<\infty$.  In Proposition \ref{prop:CZO} below, we easily show
that classical Calder\'{o}n-Zygmund operators with smooth kernels satisfy
these two conditions ---\eqref{eq:CZO1} is a simple reformulation of
the H\"{o}rmander condition \cite{Hor} and \eqref{eq:CZO2} uses the
regularity in the other variable.

The previous conditions are susceptible to generalization: in \eqref{eq:CZO1} one could have an $L^{p_0}-L^{p_0}$ estimate with $p_0\ge 1$, and the $L^1-L^\infty$ estimate in \eqref{eq:CZO2} could be replaced by an $L^{p_0}-L^{q_0}$ condition with $1\le p_0<q_0\le \infty$. This would drive us to estimates on $L^p$ in the range $(p_0,q_0)$. Still, the corresponding conditions do not involve the kernel.

Typical families of operators whose ranges of boundedness are proper subsets of $(1,\infty)$ can be built from a divergence form uniformly elliptic complex operators $ L=-\div(A\,\nabla )$ in $\re^n$. One can consider the operator $\varphi(L)$, with  bounded holomorphic functions $\varphi$ on
sectors; the Riesz transform $\nabla L^{-1/2}$; some square functions ``\`a la'' Littlewood-Paley-Stein: one, $g_{L}$, using
only  functions of $L$,  and  the other, $G_{L}$, combining functions of $L$ and the gradient operator;  estimates that control the square root $L^{1/2}$ by the gradient. These operators can be expressed in terms of the semigroup $\{e^{-t\,L}\}_{t>0}$, its gradient $\{\sqrt{t}\,\nabla e^{-t\,L}\}_{t>0}$, and their analytic extensions to some sector in $\co$. Let us stress that those operators  may not be representable with ``usable'' kernels: they are ``non-integral''.

The unweighted estimates for these operators are considered in
\cite{Aus}. The instrumental tools are two criteria for $L^p$
boundedness, valid in spaces of homogeneous type. One is  a sharper
and simpler version of a theorem by Blunck and Kunstmann \cite{BK1},
based on the Calder\'{o}n-Zygmund decomposition, where weak-type $(p,p)$
for  a  \textit{given} $p$ with  $1 \le p<p_{0}$ is presented,
knowing the weak-type $(p_0,p_0)$. We also refer to \cite{BK2} and
\cite{HM} where $L^p$ estimates are shown for the Riesz transforms of
elliptic operators for  $p<2$  starting  from the $L^2$
boundedness proved in \cite{AHLMcT}.

The second criterion is taken from \cite{ACDH}, inspired by the
good-$\lambda$ estimate in the Ph.D. thesis of one of us
\cite{Ma1,Ma2}, where strong type $(p,p)$ for \textit{some} $p>p_{0}$
is proved and applied to Riesz transforms for the Laplace-Beltrami
operators on some Riemannian manifolds. A criterion in the same
spirit for a limited range of $p$'s also appears implicitly in
\cite{CP} towards perturbation theory for linear and non-linear
elliptic equations and  more explicitly  in \cite{Shen1, Shen2}.

These results are  extended in \cite{AM1} to obtain weighted $L^p$
bounds for the operator itself, its commutators with a BMO function
and also vector-valued expressions. Using the machinery developed in
\cite{AM2} concerning off-diagonal estimates in spaces of homogeneous
type, weighted estimates for the operators above are studied in
\cite{AM3}.

Sharpness of the ranges of boundedness has been also discussed in
both the weighted and unweighted case. From \cite{Aus}, we learn that
the operators that are defined in terms of the semigroup (as
$\varphi(L)$ or $g_L$) are ruled by  the range where the semigroup
$\{e^{-t\,L}\}_{t>0}$ is uniformly bounded and/or satisfies
off-diagonal estimates (see the precise definition below). When the
gradient appears in the operators  (as in the Riesz transform $\nabla
L^{-1/2}$ or in $G_L$), the operators are bounded in the same range
where $\{\sqrt{t}\,\nabla e^{-t\,L}\}_{t>0}$ is uniformly bounded
and/or satisfies  off-diagonal estimates.

In the weighted situation, given a weight $w\in A_\infty$, one studies the previous properties for the semigroup and its gradient. Now the underlying measure is no longer $dx$ but $dw(x)=w(x)\,dx$ which is a doubling measure. Therefore, we need an appropriate definition of off-diagonal estimates in spaces of homogeneous type with the following properties: it implies uniform $L^p(w)$ boundedness, it is stable under composition, it passes from unweighted to weighted estimates and it is handy in practice. In \cite{AM2} we propose a definition only involving balls and annuli. Such definition  makes clear that there are two parameters involved,  the radius of balls  and the  parameter of the family,  linked by a scaling rule independently on the  location of the balls. The price to pay for  stability is a somewhat weak definition (in the sense that we can not be greedy in our demands). Nevertheless, it covers examples of the literature on semigroups.  Furthermore,   in spaces of homogeneous type with polynomial volume growth (that is, the measure of a ball is comparable to  a power of its radius, uniformly over centers and radii) it  coincides with some other possible definitions.  This is also the case for more general volume growth conditions, such as the one for some Lie groups with a local dimension and a dimension at infinity. Eventually,  it is  operational for proving weighted estimates  in \cite{AM3}, which was the main motivation for developing that material.

Once  it is shown in \cite{AM2} that there exist ranges where the semigroup and its gradient are uniformly bounded and/or satisfy off-diagonal estimates with respect to the weighted measure $dw(x)=w(x)\,dx$, we study the weighted estimates of the operators associates with $L$. As in the unweighted situation considered in \cite{Aus}, the ranges where the operators are bounded are ruled by either the semigroup or its gradient. To do that, one needs to apply  two criteria in a
setting  with underlying measure $dw$. Thus, we need versions of those results valid in $\re^n$ with the Euclidean distance and the measure $dw$, or more  generally, in spaces of homogeneous type (when $w\in A_\infty$ then $dw$ is doubling).

This article is a review on the subject with no proofs expect for the section dealing with Calder\'{o}n-Zygmund operators. The plan is as follows. In Section \ref{section:prelim} we give some preliminaries regarding doubling measures and  Muckenhoupt weights. In Section \ref{section:CZ} we present the two main results that generalize the Calder\'{o}n-Zygmund theory. The easy application to classical Calder\'{o}n-Zygmund operators is given with proofs. We devote Section \ref{section:off} to discuss two notions of off-diagonal estimates: one that holds for arbitrary closed sets, and another one, which is more natural in the weighted case, involving only balls and annuli. In Section \ref{sec:elliptic} we introduce the class of elliptic operators and present their off-diagonal properties. Unweighted and weighted estimates for the functional calculus, Riesz transforms and square functions associated such elliptic operators are in Section  \ref{section:applications}. The strategy to prove these results
is explained in Section \ref{section:proofs}. Finally in Section \ref{section:further} we present some further applications concerning commutators with $\BMO$ functions, reverse inequalities for square roots and also vector-valued estimates. We also give some weighted estimates for fractional operators  (see \cite{AM5}) and Riesz transforms on manifolds (see \cite{AM4}).

\section{Preliminaries}\label{section:prelim}

Given a ball $B\subset \re^n$ with radius $r(B)$ and $\lambda>0$, $\lambda\, B$ denotes the concentric ball with radius $r(\lambda \, B)= \lambda\, r(B)$.

The underlying space is the Euclidean setting  $\re^n$ equipped with the Lebesgue measure or more in general with a doubling measure $\mu$. Let us recall that $\mu$ is doubling if
$$
\mu(2\,B)\le C\,\mu(B)<\infty
$$
for every ball $B$. By iterating this expression,
there exists $D$, which is called the doubling order of $\mu$, so that $\mu(\lambda\,B)\le C_\mu\,\lambda^D\,\mu(B)$ for every $\lambda\ge 1$ and every ball $B$.

Given a ball $B$, we write $C_{j}(B)=2^{j+1}\, B\setminus 2^j\, B$
when $j\ge 2$, and $C_{1}(B)=4B$. Also we set
$$
\aver{B} h\,d\mu
=
\frac1{\mu(B)}\,\int_B h(x)\,d\mu(x),
\qquad
\aver{C_j(B)} h\,d\mu
=
\frac1{\mu(2^{j+1}B)}\,\int_{C_{j}(B)} h\,d\mu.
$$

Let  us introduce some classical classes of weights. Let $w$ be a
weight (that is, a non negative locally integrable function) on
$\RR^n$. We say that $w\in A_p$, $1<p<\infty$, if there exists a
constant $C$ such that for every ball $B\subset\re^n$,
$$
\Big(\aver{B} w\,dx\Big)\,
\Big(\aver{B} w^{1-p'}\,dx\Big)^{p-1}\le C.
$$
For $p=1$, we say that $w\in A_1$ if there is a constant $C$ such
that for every ball $B\subset \re^n$,
$$
\aver{B} w\,dx
\le
C\, w(y),
\qquad \mbox{for a.e. }y\in B.
$$
The reverse H\"{o}lder classes are defined in the following way: $w\in
RH_{q}$, $1< q<\infty$, if there is a constant $C$ such that  for any
ball $B$,
$$
\Big(\aver{B} w^q\,dx\Big)^{\frac1q}
\le C\, \aver{B} w\,dx.
$$
The endpoint $q=\infty$ is given by the condition $w\in RH_{\infty}$
whenever there is a constant $C$ such that for any ball $B$,
$$
w(y)\le C\, \aver{B} w\,dx,
\qquad \mbox{for a.e. }y\in B.
$$

The following facts are  well-known (see for instance \cite{GR, Gra}).
\begin{prop}\label{prop:weights}\

\begin{list}{$(\theenumi)$}{\usecounter{enumi}\leftmargin=1cm
\labelwidth=0.7cm\itemsep=0.15cm\topsep=.1cm
\renewcommand{\theenumi}{\roman{enumi}}
}

\item $A_1\subset A_p\subset A_q$ for $1\le p\le q<\infty$.

\item $RH_{\infty}\subset RH_q\subset RH_p$ for $1<p\le q\le \infty$.

\item If $w\in A_p$, $1<p<\infty$, then there exists $1<q<p$ such
that $w\in A_q$.

\item If $w\in RH_q$, $1<q<\infty$, then there exists $q<p<\infty$ such
that $w\in RH_p$.

\item $\displaystyle A_\infty=\bigcup_{1\le p<\infty} A_p=\bigcup_{1<q\le
\infty} RH_q $.

\item If $1<p<\infty$, $w\in A_p$ if and only if $w^{1-p'}\in
A_{p'}$.

\item If $w\in A_{\infty}$, then the measure $dw(x)=w(x)\, dx$ is a Borel doubling measure.

\end{list}
\end{prop}

\bigskip

Given $1\le p_0<q_0\le \infty$ and $w\in A_\infty$  we define the set
$$
\W_w(p_0,q_0)
=
\big\{
p: p_0<p<q_0, w\in A_{\frac{p}{p_0}}\cap
RH_{\left(\frac{q_0}{p}\right)'}
\big\}.
$$
If $w=1$, then $\W_{1}(p_0,q_0)= (p_0,q_0)$. As it is shown in \cite{AM1}, if not empty, we have
$$
\W_w(p_0,q_0)=\Big(p_0\,r_w,
\frac{q_0}{(s_w)'}\Big)
$$
where
$
r_w=\inf\{r\ge 1\, : \, w\in A_r\}$ and $
s_w=\sup\{s>1\, : \, w\in RH_s\}
$.

If the Lebesgue measure is replaced by a Borel doubling measure
$\mu$, all the above properties remain valid with the notation change \cite{ST}. A particular case is the doubling measure $dw(x)=w(x)\,dx$ with $w\in A_\infty$.

\section{Generalized Calder\'{o}n-Zygmund theory}\label{section:CZ}

As mentioned before, we have two criteria that allow us to derive the unweighted and weighted estimates. These generalize the classical Calder\'{o}n-Zygmund theory and we would like to emphasize that the conditions imposed involve the operator and its action on some functions but not its kernel.

The following result appears in \cite{BK1} in a slightly more complicated way with extra hypotheses. See \cite{Aus} and \cite{AM1} for stronger forms and more references. We notice that this result is applied to go below a given $q_0$, where it is assumed that the operator in question is \textit{a priori} bounded on $L^{q_0}(\mu)$. The proof is based on the Calder\'{o}n-Zygmund decomposition.

\begin{theor}\label{theor:small}
Let $\mu$ be a doubling Borel measure on $\re^n$,  $D$ its doubling order
and $1\le p_0< q_0\le \infty$. Suppose that $T$ is a sublinear
operator bounded on $L^{q_0}(\mu)$ and that $\{\A_B\}$ is a
family indexed by balls of linear operators acting from $L^\infty_{c}(\mu)$ into $L^{q_0}(\mu)$.  Assume that
\begin{equation}\label{small:T:I-A:variant}
\frac {1}{\mu(4\,B)}\,\int_{\re^n \setminus 4B} |T(I-\A_{B})f|\,d\mu
\lesssim
\Big(\aver{B} |f|^{p_0}\,d\mu\Big)^{\frac1{p_0}},
\end{equation}
and, for $j\ge 1$,
\begin{equation}\label{small:A}
\Big(
\aver{C_j(B)} |\A_{B}f|^{q_0}\,d\mu\Big)^{\frac1{q_0}}
\le
\alpha_j\,\Big(\aver{B} |f|^{p_0}\,d\mu\Big)^{\frac1{p_0}},
\end{equation}
for all ball $B$ and for all $f\in L^\infty_{c}(\mu)$ with $\supp f\subset B$. If $\sum_j \alpha_j\,2^{D\,j}<\infty$, then $T$ is of weak type
$(p_0,p_0)$, hence $T$ is of strong type $(p,p)$  for all
$p_0<p<q_0$. More precisely,  there exists a constant $C$ such that for all $f \in L^\infty_{c}(\mu)$
$$
\|Tf\|_{L^p(\mu)} \le C\, \| f\|_{L^p(\mu)}.
$$
\end{theor}

A stronger form of \eqref{small:T:I-A:variant}
 is with the notation above,
 \begin{equation}\label{small:T:I-A}
\Big( \aver{C_j(B)} |T(I-\A_{B})f|^{p_0}\,d\mu\Big)^{\frac1{p_0}}
\le
\alpha_j\,\Big(\aver{B} |f|^{p_0}\,d\mu\Big)^{\frac1{p_0}},
\qquad j\ge 2.
\end{equation}

Our second result is based on a good-$\lambda$ inequality. See  \cite{ACDH}, \cite{Aus} (in  the unweighted case) and \cite{AM1} for more general formulations. In contrast with Theorem \ref{theor:small}, we do not assume any \textit{a priori} estimate for $T$. However, in practice, to deal with the local term (where $f$ is restricted to $4\,B$) in \eqref{big:T:I-A}, one uses that the operator is bounded on $L^{p_0}(\mu)$. Thus, we apply this result to go above $p_0$ in the unweighted case and also to show weighted estimates.

\begin{theor} \label{theor:big}
Let $\mu$ be a doubling Borel measure on $\re^n$ and $1\le p_0<q_0\le \infty$. Let $T$ be a sublinear
operator acting on $L^{p_0}(\mu)$ and let  $\{\A_B\}$ be a
family indexed by balls of operators acting from $L^\infty_c(\mu)$ into $L^{p_{0}}(\mu)$.  Assume that
\begin{equation}\label{big:T:I-A}
\Big(\aver{B}
|T(I-\A_{B})f|^{p_0}\,d\mu\Big)^{\frac1{p_0}}
\le
\sum_{j\ge 1} \alpha_j\,
\Big( \aver{2^{j+1}\,B}
|f|^{p_0}\,d\mu\Big)^{\frac1{p_0}},
\end{equation}
and
\begin{equation}\label{big:T:A}
\Big(\aver{B} |T\A_{B}f|^{q_0}\,d\mu\Big)^{\frac1{q_0}}
\le
\sum_{j\ge 1} \alpha_j\,
\Big( \aver{2^{j+1}\,B}
|T f|^{p_0}\,d\mu\Big)^{\frac1{p_0}},
\end{equation}
for all $f\in L^\infty_c(\mu)$, all $B$ and for some $\alpha_j$ satisfying  $\sum_j \alpha_j<\infty$.
\begin{list}{$(\theenumi)$}{\usecounter{enumi}\leftmargin=1cm
\labelwidth=0.7cm\itemsep=0.15cm\topsep=.1cm
\renewcommand{\theenumi}{\alph{enumi}}}

\item If $p_0<p<q_0$, there exists a constant $C$ such that for all $f \in L^\infty_{c}(\mu)$,
$$
\|Tf\|_{L^p(\mu)} \le C\, \| f\|_{L^p(\mu)}.
$$

\item Let $p\in \W_w(p_0,q_0)$, that is,
$p_0<p<q_0$ and $w\in A_{\frac{p}{p_0}}\cap
RH_{\left(\frac{q_0}{p}\right)'}$. There is a  constant $C$ such that
for all  $f\in  L^\infty_c(\mu)$,
\begin{equation}\label{est:main-th} \|T f\|_{L^p(w)}
\le
C\, \|f\|_{L^p(w)}.
\end{equation}
\end{list}

\end{theor}

An operator acting from $A$ to $B$ is just a map from $A$ to $B$.
Sublinearity means $\vert T(f+g) \vert \le \vert Tf\vert + \vert
Tg\vert$ and $\vert T(\lambda f)\vert = \vert \lambda\vert\, \vert
Tf\vert$ for all $f,g$ and $\lambda\in \re$ or $\co$. Next, $L^p(w)$ is the
space of complex valued functions in  $ L^p(dw)$  with  $dw=w\,d\mu$.
However, all this extends to   functions  valued in a Banach space and also to functions defined on a space of homogeneous type.

Let us notice that in both results, the cases $q_0=\infty$ are understood  in the sense that the $L^{q_0}$-averages are indeed essential suprema. One can weaken \eqref{big:T:A} by adding to the right hand side   error terms such as $M(|f|^{p_0})(x)^{1/p_0}$ for any $x\in B$ (see \cite[Theorem 3.13]{AM1}).

\subsection{Application to Calder\'{o}n-Zygmund operators}

We see that the previous results allow us to reprove the unweighted and weighted estimates of classical Calder\'{o}n-Zygmund operators. We emphasize that the conditions imposed do not involve the kernels of the operators.

\begin{corol}\label{corol:CZO}
Let $T$ be a sublinear operator bounded on $L^2(\re^n)$.
\begin{list}{$(\theenumi)$}{\usecounter{enumi}\leftmargin=1cm
\labelwidth=0.7cm\itemsep=0.15cm\topsep=.1cm
\renewcommand{\theenumi}{\roman{enumi}}
}

\item Assume that, for any ball $B$ and any bounded function $f$ supported on $B$ with mean $0$, we have
\begin{equation}\label{eq:corol:CZO1}
\int_{\RR^n \setminus 4\,B} |Tf(x)|\, dx
\le C \int_B |f(x)|\, dx.
\end{equation}
Then, $T$ is of weak-type $(1,1)$ and consequently bounded on $L^p(\re^n)$ for every $1<p<2$.

\item Assume that, for any ball $B$ and any bounded function $f$ supported on $\re^n\setminus 4\,B$, we have
\begin{equation}\label{eq:corol:CZO2}
\sup_{x\in B}   |Tf(x)|
\le
C \aver{2\,B} |Tf(x)|\, dx +  C\inf_{x\in B} Mf(x).
\end{equation}
Then, $T$ is bounded on $L^p(\re^n)$ for every $2<p<\infty$.

\item  If $T$ satisfies \eqref{eq:corol:CZO1} and \eqref{eq:corol:CZO2} then,  $T$ is bounded on $L^p(w)$, for every $1<p<\infty$ and $w\in A_p$.

\end{list}

\end{corol}

\begin{proof}[Proof of $(i)$] We are going to use Theorem \ref{theor:small} with $p_0=1$ and $q_0=2$. By assumption $T$ is bounded on $L^2(\re^n)$. For every ball  $B$ we set $\A_B f(x)=\big(\textaver{B} f\,dx\big)\,\bigchi_B(x)$.  Then, as a consequence of \eqref{eq:corol:CZO1} we obtain \eqref{small:T:I-A:variant}  with $p_0=1$:
$$
\frac {1}{|4\,B|}\,\int_{\re^n \setminus 4\,B} |T(I-\A_{B})f|\,dx
\lesssim
\aver{B} |(I-\A_{B})f|\,dx
\lesssim
\aver{B} |f|\,dx.
$$
On the other hand, we observe that $\A_{B} f(x)\equiv 0$ for $x\in
C_j(B)$ and $j\ge 2$, and for $x\in C_1(B)=4\,B$ we have $|\A_{B}
f(x)|\le \textaver{B}|f|\,dx$. This shows \eqref{small:A} with
$p_0=1$ and $q_0=2$.\footnote[2]{In fact, we have \eqref{small:A}
with $p_0=1$ and $q_0=\infty$. We take $q_0=2$ since in Theorem
\ref{theor:small} we need $T$ bounded on $L^{q_0}(\re^n)$. This shows
that one can assume boundedness on $L^r(\re^n)$ for some $1<r<\infty$
(in place of $L^2(\re^n)$), and the argument goes through.}
Therefore, Theorem \ref{theor:small}
yields that $T$ is of weak-type $(1,1)$.

By Marcinkiewicz interpolation theorem, it follows that $T$ is
bounded on $L^p(\re^n)$ for $1<p<2$.
\end{proof}

\begin{proof}[Proof of $(ii)$]
We use $(a)$ of Theorem \ref{theor:big} with $p_0=2$ and $q_0=\infty$. For every ball $B$ we set $\A_B f(x)=\bigchi_{\re^n\setminus 4\,B}(x)\,f(x)$. Using that $T$ is bounded on $L^2(\re^n)$ we trivially  obtain
\begin{equation}\label{eqn:aux-CZO}
\Big(\aver{2\,B}
|T(I-\A_{B})f|^{2}\,dx\Big)^{\frac1{2}}
\lesssim
\Big( \aver{4\,B}
|f|^{2}\,dx\Big)^{\frac1{2}},
\end{equation}
which implies \eqref{big:T:I-A} with $p_0=2$.
On the other hand, \eqref{eq:corol:CZO2} and \eqref{eqn:aux-CZO} yield
\begin{align*}
\|T\A_B f\|_{L^\infty(B)}
&\lesssim
\aver{2\,B} |T\A_B\,f(x)|\, dx +  \inf_{x\in B} M(\A_B f)(x)
\\
&\le
\aver{2\,B} |T(I-\A_B)\,f(x)|\, dx+ \aver{2\,B} |T f(x)|\,dx +\inf_{x\in B} Mf(x)
\\
&
\lesssim
\Big( \aver{4\,B}
|f|^{2}\,dx\Big)^{\frac1{2}}+
\aver{2\,B} |T f(x)|\,dx+
\inf_{x\in B} Mf(x)
\\
&\le
\aver{2\,B} |T f(x)|\,dx+
\inf_{x\in B} M(|f|^2)(x)^\frac12.
\end{align*}
Thus we have obtained \eqref{big:T:A}, plus an error term $\inf_{x\in
B} M(|f|^2)(x)^\frac12$, with $q_0=\infty$ and $p_0=2$. Applying part
$(a)$ of Theorem \ref{theor:big} with the remark that follows it,
we conclude  that $T$ is bounded on $L^p(\re^n)$ for every
$2<p<\infty$.\footnote[3]{As before, if one \textit{a priori} assumes
boundedness on $L^r(\re^n)$ for some $1<r<\infty$ (in place of
$L^2(\re^n)$) the same computations hold with $r$ replacing $2$, see
the proof of $(iii)$.}

Let us observe that part $(b)$ in Theorem \ref{theor:big} yields weighted estimates: $T$ is bounded on $L^p(w)$ for every $2<p<\infty$ and $w\in A_{p/2}$.
\end{proof}

\begin{proof}[Proof of $(iii)$]
Note that $(ii)$ already gives weighted estimates. Here we improve this by assuming $(i)$, that is, by using that $T$ is bounded on $L^q(\re^n)$ for $1<q<2$.

Fixed $1<p<\infty$ and $w\in A_p$, there exists $1<r<p$ such that $w\in A_{p/r}$. Then, by $(i)$ (as we can take $r$ very close to $1$) we have that $T$ is bounded on $L^r(\re^n)$. Then, as in \eqref{eqn:aux-CZO}, we have
$$
\Big(\aver{2\,B}
|T(I-\A_{B})f|^{r}\,dx\Big)^{\frac1{r}}
\lesssim
\Big( \aver{4\,B}
|f|^{r}\,dx\Big)^{\frac1{r}}.
$$
This estimate allows us to obtain as before
$$
\|T\A_B f\|_{L^\infty(B)}
\lesssim
\aver{2\,B} |T f(x)|\,dx+
\inf_{x\in B} M(|f|^r)(x)^\frac1r.
$$
Thus we can apply part $(b)$ of Theorem \ref{theor:big} with $p_0=r$
and $q_0=\infty$ to conclude that $T$ is bounded on $L^q(u)$ for
every $r<q<\infty$ and $u\in A_{q/r}$. In particular, we have that
$T$ is bounded on $L^p(w)$.
\end{proof}

\begin{remark}\rm
We mention \cite[Theorem 3.14]{AM1} and \cite[Theorem 3.1]{Shen2} where \eqref{eq:corol:CZO2} is generalized to
$$
\Big(\aver{B}|Tf(x)|^{q_0}\,dx\Big)^\frac1{q_0}
\lesssim
\Big(\aver{2B} |Tf(x)|^{p_0}\, dx\Big)^\frac1{p_0} +  \inf_{x\in B} M\big(|f|^{p_0}\big)(x)^\frac1{p_0},
$$
for some $1\le p_0<q_0\le \infty$.
In that case, if  $T$ is bounded on $L^{p_0}(\re^n)$, proceeding as in the proofs of $(ii)$ and $(iii)$, one concludes that $T$ is bounded on $L^p(\re^n)$ for every $p_0<p<q_0$ and also on $L^p(w)$ for every $p_0<p<q_0$ and $w\in A_{p/p_0}\cap RH_{(q_0/p)'}$.
\end{remark}

\begin{remark}\label{remark:CZO-weak}\rm
To show that $T$ maps $L^1(w)$ into $L^{1,\infty}(w)$ for every $w\in A_1$ one needs to strengthen \eqref{eq:corol:CZO1}. For instance,  we can assume that \eqref{eq:corol:CZO1} holds with $dw(x)=w(x)\,dx$ in place  of $dx$ and for functions $f$ with mean value zero with respect to $dx$. In this case, we choose $\A_B$ as in $(i)$. Taking into account that $w\in A_1$ yields $\textaver{B} |f|\,dx\lesssim \textaver{B} |f|\,dw$, the proof follows the same scheme replacing everywhere (except for the definition of $\A_B$) $dx$ by $dw$. Notice that the boundedness of $T$ on $L^2(w)$ is needed, this is guaranteed by $(iii)$ as $A_1\subset A_2$.

Let us observe that we can assume that \eqref{eq:corol:CZO1} holds with $dw(x)=w(x)\,dx$ in place  of $dx$, but  for functions $f$ with mean value zero with respect to $dw$. In that case, the proof goes through by replacing everywhere $dx$ by $dw$, even in the definition of $\A_B$.

When applying this to classical operators the first approach is more natural.
\end{remark}

\begin{prop}\label{prop:CZO}
Let $T$ be a singular integral operator with kernel $K$, that is,
$$
T f(x)=\int_{\re^n} K(x,y)\,f(y)\,dy,
\qquad
x\notin\supp f,
\quad
f\in L^\infty_c,
$$
where $K$ is a measurable function defined away from the diagonal.
\begin{list}{$(\theenumi)$}{\usecounter{enumi}\leftmargin=1cm
\labelwidth=0.7cm\itemsep=0.15cm\topsep=.1cm
\renewcommand{\theenumi}{\roman{enumi}}
}

\item If $K$ satisfies the H\"{o}rmander condition
$$
\int_{|x-y|>2\,|y-y'|} |K(x,y)-K(x,y')|\,dx\le C
$$
then \eqref{eq:corol:CZO1} holds.

\item If $K$ satisfies the H\"{o}lder condition
$$
|K(x,y)-K(x',y)|\le
C\,\frac{|x-x'|^\gamma}{|x-y|^{n+\gamma}},
\qquad
|x-y|>2\,|x-x'|,
$$
for some $\gamma>0$, then \eqref{eq:corol:CZO2} holds.
\end{list}
\end{prop}

\begin{remark}\rm
Notice that in $(i)$ the smoothness is assumed with respect to the second variable and in $(ii)$ with respect to the first variable. If one assumes the stronger  H\"{o}lder condition in $(i)$, it is easy to see  that \eqref{eq:corol:CZO1} holds with $dw(x)=w(x)\,dx$ in place  of $dx$ for every $w\in A_1$. Therefore, the first approach in Remark \ref{remark:CZO-weak} yields that $T$ maps $L^1(w)$ into $L^{1,\infty}(w)$ for $w\in A_1$.
\end{remark}

\begin{proof}
We start with $(i)$. Let $B$ be a ball with center $x_B$. For every $f\in L^\infty_c(\re^n)$ with $\supp f\subset B$ and $\int_B f\,dx=0$ we obtain \eqref{eq:corol:CZO1}:
\begin{align*}
&\int_{\RR^n \setminus 4\,B} |Tf(x)|\, dx
=
\int_{\RR^n \setminus 4\,B}
\Big| \int_{B} (K(x,y)-K(x,x_B))\,f(y)\,dy\Big|\,dx
\\
&\hskip1cm
\le
\int_B |f(y)|\,\int_{|x-y|>2\,|y-x_B|} |K(x,y)-K(x,x_B)|\,dx\,dy
\lesssim
\int_B |f(y)|\,dy.
\end{align*}

We see $(ii)$. Let $B$ be a ball and $f\in L^\infty_c$  be supported on $\re^n\setminus 4\,B$. Then, for every $x\in B$ and $z\in \frac12\,B$ we have
\begin{align*}
&
|T f(x)-Tf(z)|
\le
\int_{\re^n\setminus 4\,B} |K(x,y)-K(z,y)|\,|f(y)|\,dy
\\
&
\qquad\lesssim
\sum_{j=2}^\infty \int_{C_j(B)} \frac{|x-z|^\gamma}{|x-y|^{n+\gamma}}\, |f(y)|\,dy
\lesssim
\sum_{j=2}^\infty 2^{-j\,\gamma}\,\aver{2^{j+1}\,B} |f(y)|\,dy
\lesssim
\inf_{x\in B} Mf(x).
\end{align*}
Then, for every $x\in B$ we have as desired
$$
|T f(x)|
\le
\aver{\frac12\,B} |T f(z)|\,dz
+
\aver{\frac12\,B} |Tf(x)-T f(z)|\,dz
\lesssim
\aver{2\,B} |T f(z)|\,dz
+
\inf_{x\in B} Mf(x).
$$
\end{proof}

\section{Off-diagonal estimates}\label{section:off}

We  extract from \cite{AM2}  some definitions and results (sometimes in weaker form) on  unweighted  and weighted off-diagonal estimates. See there for details and more precise statements.  Set $d(E,F)=\inf \{|x-y|\, : \, x\in E, y \in F\}$ where $E, F$ are subsets of $\RR^n$.

\begin{defin}\label{def:full}
Let $1\le p\le q \le \infty$. We say that a family $\{T_t\}_{t>0}$ of
sublinear operators satisfies  $L^p-L^q$  {full off-diagonal
estimates}, in short $T_{t}\in \fullx{p}{q}$,   if for some $c>0$,
for all closed sets $E$ and $F$, all $f$ and all $t>0$ we
have
\begin{equation}\label{eq:offLpLq}
\Big(\int_{F}|T_t (\bigchi_{E}\,  f)|^q\, dx\Big)^{\frac 1q} \lesssim  t^{- \frac 1 2 (\frac n p - \frac n q )}
\expt{-\frac{c\, d^2(E,F)}{t}} \Big( \int_{E}|f|^p\, dx\Big)^{\frac 1 p}.
\end{equation}
\end{defin}

Full off-diagonal estimates on a general space of homogenous type, or in the weighted case, are not expected since  $L^p(\mu)-L^q(\mu)$ full off-diagonal estimates when $p<q$  imply
$L^p(\mu)-L^q(\mu)$ boundedness but not $L^p(\mu)$ boundedness.
For example, the heat semigroup $e^{-t\Delta}$ on functions  for
general Riemannian manifolds with the doubling property is not $L^p-L^q$ bounded when $p<q$ unless    the measure of any ball is bounded
below by a power of  its radius.

 The following notion of off-diagonal estimates in spaces of homogeneous type  involved only balls and annuli. Here we restrict the definition of \cite{AM2} to the weighted situation, that is, for $dw=w(x)\,dx$ with $w\in A_\infty$. When $w=1$, it turns out to be equivalent to full off-diagonal estimates. Also, it passes from unweighted estimates to weighted estimates.

We set  $\dec{s}=\max\{s,s^{-1}\}$ for $s>0$.  Given a ball $B$, recall that $C_j(B)=2^{j+1}\,B\setminus
2^j\,B$ for $j\ge 2$   and if $w\in A_{\infty}$ we use the notation
$$
\aver{B} h\,dw
=
\frac1{w(B)}\,\int_B h\,dw,
\qquad
\aver{C_j(B)} h\,dw
=
\frac1{w(2^{j+1}B)}\,\int_{C_{j}(B)} h\,dw.
$$

\begin{defi}\label{defi:off-d:weights}
Given $1\le p\le q\le \infty$ and any weight $w\in A_{\infty}$, we say that a
family of sublinear operators $\{T_t\}_{t>0}$ satisfies
$L^{p}(w)-L^{q}(w)$ off-diagonal estimates on balls, in short $T_t \in\offw{p}{q}$, if there
exist $\theta_1, \theta_2>0$ and $c>0$ such that for every $t>0$ and
for any ball $B$ with radius $r$ and all $f$,
\begin{equation}\label{w:off:B-B}
\Big(\aver{B} |T_t( \bigchi_B\, f) |^{q}\,dw\Big)^{\frac 1q}
\lesssim
\dec{\frac{r}{\sqrt{t}}}^{\theta_2} \,\Big(\aver{B}
|f|^{p}\,dw\Big)^{\frac 1p};
\end{equation}
and, for all $j\ge 2$,
\begin{equation}\label{w:off:C-B}
\Big(\aver{B}|T_t( \bigchi_{C_j(B) }\, f) |^{q}\,dw\Big)^{\frac 1q}
\lesssim
2^{j\,\theta_1}\, \dec{\frac{2^j\,r}{\sqrt{t}}}^{\theta_2}\,
\expt{-\frac{c\,4^{j}\,r^2}{t}} \,
\Big(\aver{C_j(B)}|f|^{p}\,dw\Big)^{\frac 1p}
\end{equation}
and
\begin{equation}\label{w:off:B-C}
\Big(\aver{C_j(B)}|T_t( \bigchi_B\, f) |^{q}\,dw\Big)^{\frac 1q}
\lesssim
2^{j\,\theta_1}\, \dec{\frac{2^j\,r}{\sqrt{t}}}^{\theta_2}\,
\expt{-\frac{c\,4^{j}\,r^2}{t}}
\,\Big(\aver{B}|f|^{p}\,dw\Big)^{\frac 1p}.
\end{equation}
\end{defi}
Let us make some relevant comments (see \cite{AM2} for
further details and more properties).

\begin{list}{$\bullet$}{\leftmargin=.6cm
\labelwidth=0.7cm\itemsep=0.15cm\topsep=.15cm}

\item In the Gaussian factors the value of $c$ is irrelevant as long as it remains positive.

\item These definitions can be extended to complex families $\{T_z\}_{z\in \Sigma_{\theta}}$
with $t$ replaced by $|z|$ in the estimates.

\item $T_{t}$ may only be defined on a dense subspace $\D$ of $L^p$ or $L^p(w)$ ($1\le p<\infty$) that is stable by truncation by indicator functions of measurable sets (for example, $L^p\cap L^2$, $L^p(w) \cap L^2$ or $L^\infty_{c}$).

\item If $q=\infty$,
one should adapt the definitions in the usual straightforward way.

\item $L^1(w)-L^\infty(w)$  off-diagonal estimates on balls
are equivalent to pointwise  Gaussian upper bounds  for the
kernels of $T_{t}$.

\item H\"{o}lder's inequality implies $
\offw{p}{q}\subset\offw{p_1}{q_1}
$
 for all $p_1,q_1$ with $p\le p_1\le q_1\le
q$.

\item  If $T_t\in\offw{p}{p}$, then $T_t$ is uniformly bounded on $L^p(w)$.

\item This notion is stable by composition:
$T_t\in\offw{q}{r}$ and $S_t\in\offw{p}{q}$ imply $T_t\circ S_t\in
\offw{p}{r}$ when $1\le p
\le q \le r\le\infty$.

\item When $w=1$, $L^p-L^q$ off-diagonal estimates on balls  are equivalent to $L^p-L^q$ full off-diagonal estimates.

\item Given $1\le p_0<q_0\le \infty$, assume that $T_t\in\off{p}{q}$ for every $p$, $q$ with $p_0<p \le q<q_0$. Then, for all $p_0<p \le q<q_0$
    and for any $w\in A_{\frac{p}{p_0}}\cap  RH_{(\frac {q_{0}}q)'}$ we have that $T_t\in \offw{p}{q}$, equivalently, $T_t\in \offw{p}{q}$ for every $p\le q$ with $p,q\in\W_w(p_0,q_0)$.

\end{list}

\section{Elliptic operators and their off-diagonal estimates}\label{sec:elliptic}
We introduce the class of elliptic operators considered.
Let $A$ be an $n\times n$ matrix of complex and
$L^\infty$-valued coefficients defined on $\re^n$. We assume that
this matrix satisfies the following ellipticity (or \lq\lq
accretivity\rq\rq) condition: there exist
$0<\lambda\le\Lambda<\infty$ such that
$$
\lambda\,|\xi|^2
\le
\Re A(x)\,\xi\cdot\bar{\xi}
\quad\qquad\mbox{and}\qquad\quad
|A(x)\,\xi\cdot \bar{\zeta}|
\le
\Lambda\,|\xi|\,|\zeta|,
$$
for all $\xi,\zeta\in\co^n$ and almost every $x\in \re^n$. We have used the notation
$\xi\cdot\zeta=\xi_1\,\zeta_1+\cdots+\xi_n\,\zeta_n$ and therefore
$\xi\cdot\bar{\zeta}$ is the usual inner product in $\co^n$. Note
that then
$A(x)\,\xi\cdot\bar{\zeta}=\sum_{j,k}a_{j,k}(x)\,\xi_k\,\bar{\zeta_j}$.
Associated with this matrix we define the second order divergence
form operator
$$
L f
=
-\div(A\,\nabla f),
$$
which is understood in the standard weak sense as a maximal-accretive operator on $L^2(\RR^n,dx)$ with domain $\D(L)$ by means of a
sesquilinear form.

The operator $-L$ generates a $C^0$-semigroup
$\{e^{-t\,L}\}_{t>0}$ of contractions on $L^2(\RR^n,dx)$.
Define  $\vartheta\in[0,\pi/2)$ by,
$$
\vartheta
=
\sup
\big\{ \big|\arg \langle  Lf,f\rangle\big|\, : \,
f\in\mathcal{D}(L)\big\}.
$$
Then, the semigroup   has an analytic extension to a complex semigroup
$\{e^{-z\,L}\}_{z\in\Sigma_{\pi/2- \vartheta}}$ of contractions on
$L^2(\RR^n,dx)$. Here we have written
$
\Sigma_{\theta}
=
\{z\in\co^*:|\arg z|<\theta\}
$, $0<\theta<\pi$.

The families
$\{e^{-t\,L}\}_{t>0}$, $\{\sqrt t\, \nabla e^{-t\,L}\}_{t>0}$, and their analytic extensions satisfy full off-diagonal on $L^2(\re^n)$. These estimates can be extended to some other ranges that, up to endpoints, coincide with those of uniform boundedness.

We define $\widetilde \J(L)$, respectively $\widetilde \K(L)$, as the interval of those exponents $p\in[1,\infty]$ such that  $\{e^{-t\,L}\}_{t>0}$, respectively $\{\sqrt t\, \nabla e^{-t\,L}\}_{t>0}$, is a bounded set in  $\mathcal{L}(L^p(\re^n))$  (where $\mathcal{L}(X)$ is the
space of linear continuous maps on a Banach space $X$).

\begin{prop}[\cite{Aus, AM2}] \label{prop:sgfull}
Fix $m\in \NN$ and $0<\mu <\pi/2-\vartheta$.
\begin{list}{$(\theenumi)$}{\usecounter{enumi}\leftmargin=.8cm
\labelwidth=0.7cm\itemsep=0.15cm\topsep=.3cm
\renewcommand{\theenumi}{\alph{enumi}}}

\item   There exists a non empty maximal  interval of $[1,\infty]$,  denoted by $\J(L)$, such that if $p, q \in \J(L)$ with $p\le q$, then  $\{e^{-t\,L}\}_{t>0}$ and
    $\{(zL)^me^{-z\,L}\}_{z\in \Sigma_{\mu}}$ satisfy  $L^p-L^q$ full off-diagonal estimates  and are bounded sets in  $\mathcal{L}(L^p)$. Furthermore, $\J(L)\subset \widetilde \J(L)$ and $\Int\J(L)=\Int\widetilde\J(L)$.

\item There exists a non empty maximal interval of $[1,\infty]$, denoted by $\K(L)$,  such that if $p, q \in \K(L)$ with $p\le q$, then $\{ \sqrt t \, \nabla e^{-t\,L}\}_{t>0}$ and
    $\{ \sqrt z \, \nabla (zL)^me^{-z\,L}\}_{z\in \Sigma_{\mu}}$ satisfy $L^p-L^q$ full off-diagonal estimates and  are  bounded sets in  $\mathcal{L}(L^p)$.   Furthermore, $\K(L) \subset \widetilde \K(L)$ and $\Int\K(L)=\Int\widetilde\K(L)$.

\item $\K(L) \subset \J(L)$ and, for $p<2$, we have $p\in \K(L)$ if and only if $p\in \J(L)$.

\item Denote by $p_{-}(L), p_{+}(L)$ the lower and upper bounds of \,$\J(L)$ \textup{(}hence,
of \/ $\Int\widetilde\J(L)$ also\textup{)} and by $q_{-}(L),
q_{+}(L)$ those of \,$\K(L)$ \textup{(}hence, of
\,$\Int\widetilde\K(L)$ also\textup{)}. We have $p_{-}(L)=q_{-}(L)$
and $(q_{+}(L))^*\le p_{+}(L)$.

\item If $n=1$,  $\J(L) =\K(L)=[1,\infty]$.

\item If $n=2$, $\J(L)=[1,\infty]$ and $\K(L)\supset [1,q_{+}(L))$ with $q_{+}(L)>2$.

\item If $n\ge 3$,  $p_{-}(L)< \frac{2n}{n+2}$,  $p_{+}(L)> \frac{2n}{n-2}$ and $q_{+}(L)>2$.
\end{list}
\end{prop}

We have set $q^*= \frac{q\, n }{n - q}$, the Sobolev exponent of $q$ when $q<n$ and $q^*=\infty$ otherwise.

Given  $w\in A_\infty$, we define  $\widetilde \J_w(L)$, respectively $\widetilde \K_w(L)$, as the interval of those exponents $p\in[1,\infty]$ such that  the semigroup $\{e^{-t\,L}\}_{t>0}$, respectively its gradient $\{\sqrt t\, \nabla e^{-t\,L}\}_{t>0}$, is uniformly bounded on  $L^p(w)$. As in Proposition \ref{prop:sgfull}  uniform boundedness and weighted off-diagonal estimates on balls hold essentially in the same ranges.

\begin{prop} [\cite{AM2}]\label{prop:sg-w:extension}
Fix $m\in \NN$ and $0<\mu <\pi/2-\vartheta$.  Let $w \in A_{\infty}$.
\begin{list}{$(\theenumi)$}{\usecounter{enumi}\leftmargin=.8cm
\labelwidth=0.7cm\itemsep=0.15cm\topsep=.15cm
\renewcommand{\theenumi}{\alph{enumi}}}
\item  Assume $\W_{w}\big(p_{-}(L), p_{+}(L)\big)\ne \emptyset$. There is a maximal interval of $[1,\infty]$,  denoted by $\J_w(L)$, containing $\W_{w}\big(p_{-}(L), p_{+}(L)\big)$, such that  if $p, q \in \J_{w}(L)$ with $p\le q$, then $\{e^{-t\,L}\}_{t>0}$ and  $\{(zL)^me^{-z\,L}\}_{z\in \Sigma_{\mu}}$ satisfy $L^p(w)-L^q(w)$ off-diagonal estimates on balls and are  bounded sets in  $\mathcal {L} (L^p(w))$. Furthermore, $\J_{w}(L)\subset \widetilde\J_{w}(L)$ and $\Int\J_{w}(L)=\Int\widetilde\J_{w}(L)$.

\item Assume $\W_{w}\big(q_{-}(L), q_{+}(L)\big)\ne \emptyset$. There exists a maximal interval of $[1,\infty]$, denoted by $\K_w(L)$, containing $\W_{w}\big(q_{-}(L), q_{+}(L)\big)$ such that if $p, q \in \K_{w}(L)$ with $p\le q$, then $\{ \sqrt t \, \nabla e^{-t\,L}\}_{t>0}$ and
    $\{ \sqrt z \, \nabla (zL)^me^{-z\,L}\}_{z\in \Sigma_{\mu}}$ satisfy $L^p(w)-L^q(w)$ off-diagonal estimates on balls and are bounded sets in $\mathcal{L}(L^p(w))$.  Furthermore, $\K_{w}(L)\subset \widetilde \K_{w}(L)$ and $\Int\K_{w}(L)=\Int\widetilde\K_{w}(L)$.

\item  Let $n\ge 2$. Assume $ \W_{w}\big(q_{-}(L), q_{+}(L)\big)\ne \emptyset$. Then $\K_{w}(L)\subset \J_{w}(L)$. Moreover,   $\inf \J_{w}(L) = \inf \K_{w}(L)$ and
$(\sup \K_{w}(L))^*_w \le  \sup \J_{w}(L) $.
\item If $n=1$,  the intervals $\J_{w}(L)$ and $\K_{w}(L)$ are the same and contain $(r_{w},\infty]$ if $w\notin A_{1}$ and are equal to  $[1,\infty]$ if $w\in A_{1}$.
\end{list}
\end{prop}

We have set $q^*_w= \frac{q\, n \, r_{w}}{n\, r_{w} - q}$ when $q<n\, r_{w}$ and $q^*_{w}=\infty$ otherwise. Recall that $r_{w}=\inf\{r\ge 1\, : \, w\in A_{r}\}$ and
also that $s_{w}=\sup\{s>1\, : \, w\in RH_{s}\}$.

Note that  $\W_{w}\big(p_{-}(L), p_{+}(L)\big)\ne \emptyset$ means
$\frac{p_{+}(L)}{p_{-}(L)}  >     r_{w}(s_{w})'$.  This is a
compatibility condition between $L$ and $w$. Similarly,
$\W_{w}\big(q_{-}(L), q_{+}(L)\big)\ne \emptyset $ means
$\frac{q_{+}(L)}{q_{-}(L)}  >  r_{w}(s_{w})'$, which is a more
restrictive condition on $w$ since $q_{-}(L)=p_{-}(L)$ and
$(q_{+}(L))^*\le p_{+} (L)$.

In the case of real operators, $\J(L)=[1,\infty]$ in all dimensions
because the kernel $e^{-t\, L}$ satisfies a pointwise Gaussian upper
bound.  Hence $\W_{w}\big(p_{-}(L), p_{+}(L)\big)= (r_{w},\infty)$.
If $w\in A_{1}$, then one has that $\J_{w}(L)=[1,\infty]$. If
$w\notin A_{1}$, since the kernel is also positive and satisfies a
similar pointwise lower bound, one has $\J_{w}(L) \subset
(r_{w},\infty] $.  Hence, $\Int \J_{w}(L)= \W_{w}\big(p_{-}(L),
p_{+}(L)\big)$.

The situation may change for complex operators. But we lack of examples to say whether or not
  $\J_{w}(L)$ and $ \W_{w}\big(p_{-}(L), p_{+}(L)\big)$  have  different endpoints.

\begin{remark}\label{remark:inf-gen}
\rm
Note that by density of $L^\infty_{c}$ in the spaces $L^p(w)$ for
$1\le p <\infty$, the various extensions of $e^{-z\, L}$ and $\sqrt{z}\,\nabla
e^{-z\, L}$ are all consistent. We keep the above notation to denote
any such extension.  Also, we showed in \cite{AM2} that as long as
$p\in \J_{w}(L)$ with $p\ne\infty$, $\{e^{-t\, L}\}_{t>0}$ is strongly
continuous on $L^p(w)$, hence it has an
infinitesimal generator in $L^p(w)$, which is of type $\vartheta$.
\end{remark}

\section{Applications}\label{section:applications}

In this section we apply the generalized Calder\'{o}n-Zygmund theory presented above to obtain weighted estimates for operators that are associated with $L$. The off-diagonal estimates on balls introduced above are one of the main tools.

Associated with $L$ we have the four numbers $p_{-}(L)=q_{-}(L)$ and $p_{+}(L)$, $q_{+}(L)$.  We  often drop $L$ in the notation: $p_{-}=p_{-}(L)$, $\dots$. Recall that the semigroup and its analytic extension are uniformly bounded and satisfy full off-diagonal estimates (equivalently, off-diagonal estimates on balls) in the interval $\Int \J(L)=\Int \widetilde{\J}(L)=(p_-,p_+)$. Up to endpoints, this interval is maximal for these properties.
Analogously, the gradient of the semigroup is ruled by the interval $\Int \K(L)=\Int \widetilde{\K}(L)=(q_-,q_+)$.

Given  $w\in A_{\infty}$, if $\W_w\big(p_-,p_+\big)\neq\emptyset$, then the open interval $\Int \J_w(L)$ contains $\W_w\big(p_-,p_+\big)$ and characterizes (up to endpoints) the uniform $L^p(w)$-boundedness and the weighted off-diagonal estimates on balls of the semigroup and its analytic extension. For the gradient, we assume that  $\W_w\big(q_-,q_+\big)\neq\emptyset$ and the corresponding maximal interval is
$\Int\K_w(L)$.

\subsection{Functional calculi}

Let  $\mu\in(\vartheta,\pi)$   and
 $\varphi$ be a holomorphic function in $\Sigma_{\mu}$  with the following decay
\begin{equation}\label{decay:varphi}
|\varphi(z)|
\le
c\,|z|^s\,(1+|z|)^{-2\,s},
\qquad
z\in\Sigma_\mu,
\end{equation}
for some $c$, $s>0$. Assume that $\vartheta<\theta<\nu<\mu<\pi/2$. Then we have
\begin{equation}\label{phi-L}
\varphi(L)
=
\int_{\Gamma_+} e^{-z\,L}\,\eta_+(z)\,dz +\int_{\Gamma_-}
e^{-z\,L}\,\eta_-(z)\,dz,
\end{equation}
where $\Gamma_{\pm}$ is the half ray $\re^+\,e^{\pm
i\,(\pi/2-\theta)}$,
\begin{equation}\label{phi-L:eta}
\eta_{\pm}(z)
=
\frac1{2\,\pi\,i}\,\int_{\gamma_{\pm}}
e^{\zeta\,z}\,\varphi(\zeta)\,d\zeta,
\qquad
z\in\Gamma_{\pm},
\end{equation}
with $\gamma_{\pm}$ being the half-ray $\re^+\,e^{\pm i\,\nu}$ (the
orientation of the paths is not needed in what follows so we do not
pay attention to it). Note that
$
|\eta_{\pm}(z)| \lesssim \min ( 1, |z|^{-s-1})$ for $z\in\Gamma_{\pm}
$,
hence the representation \eqref{phi-L} converges in norm in $\mathcal{L}(L^2)$.
  Usual arguments show the functional property
$\varphi(L)\,\psi(L)=(\varphi\,\psi)(L)$ for two such functions
$\varphi,\psi$.

Any $L$ as above is maximal-accretive and so it has a bounded
holomorphic functional calculus on $L^2$. Given any angle
$\mu\in(\vartheta,\pi)$:
\begin{list}{$(\theenumi)$}{\usecounter{enumi}\leftmargin=.8cm
\labelwidth=0.7cm\itemsep=0.15cm\topsep=.3cm
\renewcommand{\theenumi}{\alph{enumi}}}

\item   For any function $\varphi$, holomorphic and
bounded in $\Sigma_\mu$, the operator $\varphi(L)$ can be defined and is bounded on
$L^2$ with
$$
\|\varphi(L)f\|_{2}
\le
C\,\|\varphi\|_{\infty}\,\|f\|_{2}
$$
where $C$  only depends on $\vartheta$ and $\mu$.

\item For any sequence $\varphi_{k}$ of bounded and holomorphic functions
on $\Sigma_{\mu}$ converging uniformly  on compact subsets of
$\Sigma_{\mu}$ to $\varphi$, we have that $\varphi_{k}(L)$ converges
strongly to $\varphi(L)$ in $\mathcal{L}(L^2)$.

\item The product rule $\varphi(L)\,\psi(L)=(\varphi\,\psi)(L)$ holds for  any two
bounded and holomorphic  functions $\varphi,\psi$  in $\Sigma_{\mu}$.
\end{list}

Let us point out that for more general holomorphic functions (such as
powers), the operators $\varphi(L)$ can be defined as unbounded
operators.

Given a functional Banach space $X$, we say that $L$ has a bounded
holomorphic functional calculus on $X$ if for any
$\mu\in(\vartheta,\pi)$, and for any $\varphi$ holomorphic and satisfying \eqref{decay:varphi} in
$\Sigma_\mu$, one has
\begin{equation}
\label{eq:fcX}
\|\varphi(L)f\|_{X}
\le
C\,\|\varphi\|_{\infty}\,\|f\|_{X},
\qquad
f\in X\cap L^2,
\end{equation}
where $C$ depends only on $X$, $\vartheta$ and $\mu$ (but not on the decay of $\varphi$).

If  $X=L^p(w)$ as below,  then \eqref{eq:fcX} implies that
$\varphi(L)$ extends to a bounded operator on $X$ by density. That
$(a)$, $(b)$ and $(c)$ hold with $L^2$ replaced by $X$ for all
bounded holomorphic functions in $\Sigma_{\mu}$, follow from the
theory in \cite{Mc} using the fact that on those $X$, the semigroup
$\{e^{-t\, L}\}_{{t>0}}$ has an infinitesimal generator which is of
type $\vartheta$ (see Remark \ref{remark:inf-gen}).

\begin{theor}[\cite{BK1, Aus}]\label{theor:B-K:Aus}
If $p\in \Int\J(L)$ then  $L$ has a bounded holomorphic functional calculus on $L^p(\re^n)$. Furthermore, this range is sharp up to endpoints.
\end{theor}

The weighted version of this result is presented next. We mention
\cite{Ma1} where similar weighted estimates are proved under kernel upper
bounds assumptions.

\begin{theor}[\cite{AM3}]\label{theor:B-K:weights}
Let $w\in A_\infty$ be such that $\W_w\big(p_-(L),p_+(L)\big)\neq
\emptyset$. Let  $p\in \Int\J_{w}(L)$ and $\mu\in(\vartheta,\pi)$. For any $\varphi$ holomorphic   on $\Sigma_\mu$  satisfying \eqref{decay:varphi}, we have
\begin{equation}
\label{eq:fcw}
\|\varphi(L)f\|_{L^p(w)}
\le
C\,\|\varphi\|_{\infty}\,\|f\|_{L^p(w)},
\qquad
f\in L^\infty_{c},
\end{equation}
with $C$ independent of $\varphi$ and $f$. Hence, $L$ has a bounded holomorphic functional
calculus on $L^p(w)$.
\end{theor}

\begin{remark}\rm Fix $w\in A_{\infty}$  with $\W_w\big(p_-(L),p_+(L)\big)\neq
\emptyset$. If $1<p<\infty$ and  $L$ has a bounded holomorphic
functional calculus on $L^p(w)$, then $p\in \widetilde\J_{w}(L)$.
Indeed,  take $\varphi(z)=e^{-z}$. As
$\Int\widetilde\J_{w}(L)=\Int\J_{w}(L)$ by Proposition
\ref{prop:sgfull}, this shows that the range obtained in the
theorem is optimal up to endpoints. \end{remark}

\subsection{Riesz transforms}

The Riesz transforms associated to $L$ are $\partial_{j}L^{-1/2}$, $1\le j \le n$. Set   $\nabla L^{-1/2}=(\partial_{1}L^{-1/2}, \ldots, \partial_{n}L^{-1/2})$. The
solution of the Kato conjecture \cite{AHLMcT} implies that this operator extends boundedly to $L^2$.  This allows the representation
\begin{equation}
\label{eq:RT} \nabla L^{-1/2}f = \frac{1}{\sqrt \pi}\int_{0}^\infty
\sqrt t\,\nabla e^{-t\, L}f\, \frac{dt}{t},
\end{equation}
in which the integral converges  strongly in $L^2$ both at $0$ and
$\infty$ when $f\in L^2$. The $L^p$ estimates for this operator are characterized in \cite{Aus}.

\begin{theor}[\cite{Aus}]\label{theor:Riesz-Auscher}
 $p\in \Int\K(L)$ if and only if $\nabla L^{-1/2}$ is bounded on $L^p(\re^n)$. \end{theor}

In the weighted case we have the following analog.

\begin{theor}[\cite{AM3}]\label{theor:ext-RT}
Let $w\in A_\infty$ be such that $\W_w\big(q_-(L),q_+(L)\big)\neq
\emptyset$. For all $p\in
\Int \K_w(L)$ and $ f \in L^\infty_{c}$,
\begin{equation}
\label{eq:Riesz-w} \|\nabla L^{-1/2} f\|_{L^p(w)} \le
C\,
\|f\|_{L^p(w)}.
\end{equation}
Hence,  $\nabla L^{-1/2}$ has a bounded extension to  $L^p(w)$.
\end{theor}

For a discussion on sharpness issues concerning this result, the reader is referred to \cite[Remark 5.5]{AM3}.

\subsection{Square functions}
We define the square functions for $x\in \RR^n$ and $f\in L^2$,
\begin{eqnarray*}
g_L f(x)
&=&
\Big(
\int_0^\infty |(t\,L)^{1/2}\,e^{-t\,L}f(x)|^2\,\frac{dt}{t}
\Big)^{\frac12},
\\[0.1cm]
G_L f(x)
&=&
\Big(
\int_0^\infty |\sqrt{t}\,\nabla e^{-t\,L}f(x)|^2\,\frac{dt}{t}
\Big)^{\frac12}.
\end{eqnarray*}
These square functions satisfy the following unweighted estimates.
\begin{theor}[\cite{Aus}]\label{theor:Aus-square}
\null\ \null
\begin{list}{$(\theenumi)$}{\usecounter{enumi}\leftmargin=.8cm
\labelwidth=0.7cm\itemsep=0.15cm
\renewcommand{\theenumi}{\alph{enumi}}}

\item If $p\in \Int \J(L)$ then for all $f\in
L^p\cap L^2$, $$
\|g_L f\|_{p}\sim \|f\|_{p}.
$$
Furthermore, this range is sharp up to endpoints.

\item  If $p\in \Int \K(L)$ then for all $f\in
L^p\cap L^2$,
$$
\|G_L f\|_{p}\sim \|f\|_{p}.
$$
Furthermore, this range is sharp up to endpoints.
\end{list}
\end{theor}

In this statement, $\sim$ can be replaced by $\lesssim$: the square
function estimates for $L$ (with $\lesssim$) automatically imply the reverse
ones for $L^*$. The part concerning $g_{L}$ can be obtained using an abstract
result of  Le Merdy \cite{LeM} as a consequence of the bounded
holomorphic functional calculus on $L^p$.  The method in \cite{Aus}
is direct. We remind the reader that in \cite{Ste}, these inequalities for $L=-\Delta$ were proved differently and the boundedness  of $G_{-\Delta}$ follows from that of $g_{-\Delta}$ and of the Riesz transforms $\partial_{j}(-\Delta)^{-1/2}$ (or vice-versa) using the commutation between $\partial_{j}$ and $e^{t\, \Delta}$. Here, no such thing is possible.

We have the following weighted  estimates for square functions.

\begin{theor}[\cite{AM3}]\label{theor:square:weights}
Let $w\in A_\infty$.
\begin{list}{$(\theenumi)$}{\usecounter{enumi}\leftmargin=.8cm
\labelwidth=0.7cm\itemsep=0.15cm\topsep=.3cm
\renewcommand{\theenumi}{\alph{enumi}}}
\item[(a)] If $\W_w\big(\,p_-(L),p_+(L)\big)\neq \emptyset$ and $p\in \Int
\J_w(L)$ then  for
all $f\in L^\infty_{c}$ we have
$$
\|g_L f\|_{L^p(w)}\lesssim \|f\|_{L^p(w)}.
$$
\item[(b)] If $\W_w\big(\,q_-(L),q_+(L)\big)\neq \emptyset$ and $p\in \Int
\K_w(L)$ then  for
all $f\in L^\infty_{c}$ we have
$$
\|G_L f\|_{L^p(w)}\lesssim \|f\|_{L^p(w)}.
$$
\end{list}
\end{theor}

We also get reverse weighted square function estimates as follows.

\begin{theor}[\cite{AM3}]\label{theor:reverse-square:weights}
Let $w\in A_\infty$.
\begin{list}{$(\theenumi)$}{\usecounter{enumi}\leftmargin=.8cm
\labelwidth=0.7cm\itemsep=0.15cm\topsep=.2cm
\renewcommand{\theenumi}{\alph{enumi}}}

\item If $ \W_w\big(\,p_-(L),p_+(L)\big)\neq \emptyset$ and $p\in \Int\J_{w}(L)$ then
$$
\|f\|_{L^p(w)}\lesssim\|g_L f\|_{L^p(w)},
\qquad f\in L^p(w)\cap L^2.
$$

\item If  $r_{w}<p<\infty$,
$$
\|f\|_{L^p(w)}
\lesssim
\|G_L f\|_{L^p(w)},
\qquad f\in L^p(w)\cap L^2.
$$
\end{list}

\end{theor}

\begin{remark}
\rm Let us observe that $\Int\J_{w}(L)$ is the sharp range,  up to endpoints, for
$\|g_{L}f\|_{L^p(w)} \sim \|f\|_{L^p(w)}$. Indeed, we have $g_{L}(e^{- t\, L} f) \le g_{L}f$  for all
$t>0$. Hence,  the equivalence implies the uniform
$L^p(w)$ boundedness of $e^{-t\, L}$, which implies $p\in
\widetilde\J_{w}(L)$ (see Proposition \ref{prop:sg-w:extension}). Actually,  $\Int\J_{w}(L)$ is also the sharp range up to endpoints for the inequality
$\|g_{L}f\|_{L^p(w)} \lesssim \|f\|_{L^p(w)}$. It suffices to adapt the interpolation procedure in \cite[Theorem 7.1, Step 7]{Aus}.
Similarly, this interpolation procedure also shows that  $\Int\K_{w}(L)$ is sharp up to endpoints for  $\|G_{L}f\|_{L^p(w)} \lesssim \|f\|_{L^p(w)}$.
\end{remark}

\section{About the proofs}\label{section:proofs}

They follow a  general scheme.  First,
we choose $\A_B=I-(I-e^{-r^2\,L})^m$ with  $r$ the radius of $B$ and $m\ge 1$ sufficiently
large and whose value changes in each situation.

A first application of  Theorem \ref{theor:small} and Theorem \ref{theor:big} yield
unweighted estimates, and weighted estimates in a first range. This requires to prove \eqref{small:T:I-A:variant} (or the stronger \eqref{small:T:I-A}), \eqref{small:A}, \eqref{big:T:I-A} and \eqref{big:T:A}  with  measure $dx$, using the full off-diagonal estimates  of Proposition \ref{prop:sgfull}.

Then, having fixed $w$, a second application of  Theorems \ref{theor:small} and Theorems \ref{theor:big} yield
 weighted estimates in the largest range. This requires to prove \eqref{small:T:I-A:variant} (or the stronger \eqref{small:T:I-A}), \eqref{small:A}, \eqref{big:T:I-A} and \eqref{big:T:A}  with  measure $dw$,  using the off-diagonal estimates  on balls of Proposition \ref {prop:sg-w:extension}.

 There are technical  difficulties depending on whether operators commute or not with the semigroup. Full details are in \cite{AM3}

\section{Further results}\label{section:further}

We present some additional results obtained in \cite{AM3}, \cite{AM4}, \cite{AM5}.

\subsection{Commutators}

Let $\mu$ be a doubling measure in $\re^n$. Let $b\in \BMO(\mu)$ (BMO
is for bounded mean oscillation), that is,
$$
\|b\|_{\BMO(\mu)}
=
\sup_B \aver{B} |b-b_B| d\mu
<\infty,
$$
where the supremum is taken over balls and $b_B$ stands for the
$\mu$-average of $b$ on $B$. When $d\mu=dx$ we simply write $\BMO$.
If $w\in A_\infty$ (so $dw$ is a
doubling measure) then the reverse H\"{o}lder property yields that
$\BMO(w)=\BMO$ with equivalent norms.

For $T$ a  sublinear operator, bounded in some  $L^{p_0}(\mu)$,
$1\le p_{0}\le \infty$, $b\in \BMO$,  $k\in \mathbb{N}$, we define
the $k$-th order commutator
$$
T_b^k f(x)=T\big((b(x)-b)^k\,f\big)(x),
\qquad f\in L^\infty_c(\mu),\ x\in \re^n.
$$
Note that $T_b^0=T$ and that
$T_b^k f(x)$
is well-defined almost everywhere  when $f\in L^\infty_c(\mu)$. If $T$ is  linear
it can be alternatively defined by
recurrence: the first order commutator is
$$
T_b^1f(x)=[b,T]f(x)= b(x)\,T f(x)- T(b\,f)(x)
$$
and for $k\ge 2$,  the $k$-th order commutator is given by
$T_b^k=[b,T_b^{k-1}]$.

\begin{theor}[\cite{AM1}]
Let $k\in \NN$ and $b\in \BMO(\mu)$.
\begin{list}{$(\theenumi)$}{\usecounter{enumi}\leftmargin=.8cm
\labelwidth=0.7cm\itemsep=0.15cm\topsep=.2cm
\renewcommand{\theenumi}{\alph{enumi}}}

\item  Assume the conditions of Theorem \ref{theor:small} with \eqref{small:T:I-A:variant}  replaced by the stronger condition \eqref{small:T:I-A}. Suppose that $T$ and $T_b^m$ for $m=1,\dots,k$ are bounded on $L^{q_0}(\mu)$ and that $\sum_j \alpha_j\,2^{D\,j}\,j^k<\infty$. Then for all $p_0<p<q_0$,
    $$
    \|T_b^k f\|_{L^p(\mu)} \le C\,  \|b\|_{\BMO(\mu)}^k\, \|f\|_{L^p(\mu)}.
    $$

\item Assume the conditions of Theorem \ref{theor:big}. If $\sum_{j} \alpha_j\,j^k<\infty$, then for all $p_0<p<q_0$, $w\in A_{\frac{p}{p_0}}\cap RH_{\left(\frac{q_0}{p}\right)'}$,
    $$
    \|T_b^k f\|_{L^p(w)}
    \le
    C\, \|b\|_{\BMO(\mu)}^k\,\|f\|_{L^p(w)}.
    $$
\end{list}
\end{theor}

With these results in hand, we obtain weighted estimates for the commutators of the previous operators.

\begin{theor}[\cite{AM3}]\label{theor:appl:comm}
Let $w\in A_\infty$, $k\in \NN$ and $b\in \BMO$. Assume one of the
following conditions:
\begin{list}{$(\theenumi)$}{\usecounter{enumi}\leftmargin=.8cm
\labelwidth=0.7cm\itemsep=0.15cm\topsep=.2cm
\renewcommand{\theenumi}{\alph{enumi}}}

\item $T=\varphi(L)$ with $\varphi$   bounded holomorphic
on $\Sigma_\mu$,
$\W_w\big(p_-(L),p_+(L)\big)\neq\emptyset$ and $p\in\Int \J_w(L)$.

\item $T=\nabla\,L^{-1/2}$,
$\W_w\big(q_-(L),q_+(L)\big)\neq\emptyset$ and $p\in\Int \K_w(L)$.

\item $T=g_L$, $\W_w\big(p_-(L),p_+(L)\big)\neq\emptyset$ and
$p\in\Int \J_w(L)$.

\item $T=G_L$, $\W_w\big(q_-(L),q_+(L)\big)\neq\emptyset$ and
$p\in\Int \K_w(L)$.
\end{list}
Then, for every for $ f\in L_c^\infty(\re^n)$, we have
$$
\|T_b^k f\|_{L^p(w)}
\le C\, \|b\|_{\BMO}^k\,\|f\|_{L^p(w)},
$$
where $C$ does not depend on $f$, $b$, and  is proportional to $\|\varphi\|_{\infty}$ in case $(a)$.
\end{theor}

Let us mention that, under kernel upper bounds assumptions,
unweighted estimates for commutators in case $(a)$ are
obtained in  \cite{DY1}.

\subsection{Reverse inequalities for square roots}

The method described above  can be used to consider estimates opposite to \eqref{eq:Riesz-w}. In the unweighted case, \cite{Aus} shows that if  $f\in\cals$ and $p$ is such that $\max\big\{1,\frac{n\,p_{-}(L)}{n+p_{-}(L)}\big\}<p<  p_+(L)$, then
$$
\|L^{1/2}f\|_{p} \lesssim \|\nabla f\|_{p}.
$$
The weighted counterpart of this estimate is considered in \cite{AM3}. If  $w\in A_\infty$ with $\W_w\big(\,p_{-}(L),p_{+}(L)\big)\ne
\emptyset$, then
\begin{equation}
\label{eq:reverseRiesz:w}
\|L^{1/2}f\|_{L^p(w)} \lesssim \|\nabla f\|_{L^p(w)},
\qquad f\in\cals,
\end{equation}
for all  $p$ such that
$\max\big\{\,r_{w}\,,\,
\frac{n\, r_{w}\, \widehat p_{-}(L)}{\, n\, r_{w}+ \widehat p_{-}(L)}\big\}<p< \widehat p_+(L)$, where $r_{w}=\inf\{r\ge 1\ : \ w \in A_{r} \}$, and $\widehat p_{-}(L)$, $\widehat p_{+}(L)$ are the endpoints of $\J_w(L)$, that is,  $\big(\widehat p_{-}(L), \widehat p_{+}(L)\big)=\Int \J_w(L)$.

Define $\dot W^{1,p}(w)$ as the completion of $\cals$ under
the semi-norm $\|\nabla f\|_{L^p(w)}$. Arguing as in \cite{AT} (see
\cite{Aus}) combining Theorem   \ref{theor:ext-RT}  and \eqref{eq:reverseRiesz:w}, it follows that $L^{1/2}$
extends to an isomorphism from $\dot W^{1,p}(w)$ into $L^p(w)$
for all $p\in \Int\K_{w}(L)$ with $p>r_{w}$, provided $\W_w\big(\,q_{-}(L),q_{+}(L)\big)\ne \emptyset$.

\subsection{Vector-valued estimates}
In \cite{AM1}, by using an extrapolation result ``\`a la Rubio de Francia'' for the classes of weights $A_{\frac{p}{p_0}}\cap
RH_{\left(\frac{q_0}{p}\right)'}$, it follows automatically from
Theorem \ref{theor:big}, part $(b)$, that for every $p_{0}<p,r<q_{0}$  and $w\in A_{\frac{p}{p_0}}\cap
RH_{\left(\frac{q_0}{p}\right)'}$, one has
\begin{equation}\label{T:v-v}
\Big\|
\Big(\sum_k |T f_k|^r\Big)^{\frac1r}
\Big\|_{L^p(w)}
\lesssim
C\,\Big\|
\Big(\sum_k | f_k|^r\Big)^{\frac1r}
\Big\|_{L^p(w)}.
\end{equation}

As a consequence, one can show weighted vector-valued estimates for the previous operators (see \cite{AM3} for more details). Given $w\in A_\infty$, we have
\begin{list}{$\bullet$}{\leftmargin=.6cm
\labelwidth=0.7cm\itemsep=0.15cm\topsep=.3cm}

\item If $\W_w\big(\,p_{-}(L),p_{+}(L)\big)\ne
\emptyset$, and  $T=\varphi(L)$ ($\varphi$ bounded holomorphic in an appropriate sector) or $T=g_L$ then \eqref{T:v-v} holds for all $p,r\in\Int\J_w(L)$

\item If $\W_w\big(\,q_{-}(L),q_{+}(L)\big)\ne
\emptyset$, and  $T=\nabla L^{-1/2}$ or  $T=G_L$ then \eqref{T:v-v} holds for all $p,r\in\Int\J_w(L)\cap (r_w,\infty)$.
\end{list}

\subsection{Maximal regularity}

Other vector-valued inequalities of interest are
\begin{equation}\label{eq:MR}
\Big\|
\Big(\sum_{1\le k\le N} |e^{-\zeta_{k}L} f_k|^2\Big)^{\frac12}
\Big\|_{L^q(w)}
\le
C\,\Big\|
\Big(\sum_{1\le k\le N} | f_k|^2\Big)^{\frac12}
\Big\|_{L^q(w)}
\end{equation}
for  $\zeta_{k}\in \Sigma_{\alpha}$ with $0<\alpha<\pi/2-\vartheta$  and $f_{k}\in L^p(w)$ with a constant $C$ independent of $N$, the choice of the $\zeta_{k}$'s and the $f_{k}$'s.  We restrict to $1<q<\infty$ and $w \in  A_{\infty}$. By \cite[Theorem 4.2]{W}, we know that the existence of such a constant is  equivalent to the maximal $L^p$-regularity of  the generator $-A$ of $e^{-tL}$ on $L^q(w)$ with one/all $1<p<\infty$, that is the existence of a constant $C'$ such that for all $f\in {L^p((0,\infty), L^q(w))}$ the solution $u$ of the parabolic problem on $\RR^n\times (0,\infty)$,
$$
u'(t) +Au(t)=f(t), \qquad u(0)=0,
$$
with
$$\|u'\|_{L^p((0,\infty), L^q(w))}  + \|Au\|_{L^p((0,\infty), L^q(w))} \le C'\,\|f\|_{L^p((0,\infty), L^q(w))}.
$$

\begin{prop}[\cite{AM3}]
Let $w\in A_\infty$ be such that $\W_w\big(p_-(L),p_+(L)\big)\neq
\emptyset$. Then for any $q\in \Int\J_{w}(L)$, \eqref{eq:MR} holds with $C=C_{q,w, L}$ independent of
$N$, $\zeta_{k}$,$f_{k}$.
\end{prop}

This result follows from an abstract result of Kalton-Weis
\cite[Theorem 5.3]{KW}  together with   the bounded holomorphic
functional calculus of $L$ on  those $L^q(w)$  that we established
in Theorem  \ref{theor:B-K:weights}. However, the proof in \cite{AM3} uses extrapolation and $\ell^2$-valued versions of Theorems \ref{theor:small} and \ref{theor:big}. Note
that $q=2$ may not be contained in $\Int\J_{w}(L)$ and the
interpolation method of \cite{BK2} may not work here.

\subsection{Fractional operators}

The fractional operators associated with $L$ are formally given by, for
$\alpha>0$,
$$
L^{-\alpha/2}
=
\frac1{\Gamma(\alpha/2)}\,\int_0^\infty
t^{\alpha/2}\,e^{-t\,L}\,\frac{dt}{t}.
$$

\begin{theor}[\cite{Aus}]\label{theor:Aus:fract}
Let $p_-<p<q<p_+$ and $\alpha/n=1/p-1/q$. Then $L^{-\alpha/2}$ is
bounded from $L^p(\re^n)$ to $L^q(\re^n)$.
\end{theor}

\begin{remark}\label{remark:gauss-bounds}\rm
A special case of this result with $p_-=1$ and $p_+=\infty$  is when
 $L=-\Delta$ as one has that
$L^{-\alpha/2}=I_\alpha$, the classical Riesz potential whose
kernel is $c\,|x|^{-(n-\alpha)}$. If one has a Gaussian kernel bounds, then $|L^{-\alpha/2} f|\lesssim   I_\alpha(|f|)$ and  the result follows at once from the well known estimates for $I_\alpha$.
\end{remark}

\begin{theor}[\cite{AM5}]\label{theor:main:fract}
Let $p_-<p<q<p_+$ and $\alpha/n=1/p-1/q$. Then $L^{-\alpha/2}$ is
bounded from $L^p(w^p)$ to $L^q(w^q)$ for every $w\in A_{1+\frac1{p_-}-\frac1{p}}\cap RH_{q\,(\frac{p_+}{q})'}$. Furthermore, for every $k\in\NN$ and  $b\in\BMO$, we have that
$(L^{-\alpha/2})_b^k$ ---the $k$-th order commutator of $L^{-\alpha/2}$--- satisfies the same estimates.
\end{theor}

The proof of this result is based on a version of Theorem \ref{theor:big} adapted to the case of fractional operators and involving fractional maximal functions.

\begin{remark}\rm
In the classical case of the  commutator with the Riesz potential, unweighted estimates were considered in \cite{Chan}. Weighted estimates  were established in \cite{ST} by means
of extrapolation. Another proof based on a good-$\lambda$ estimate
was given in \cite{CF}. For $k=1$ and elliptic operators $L$ with
Gaussian kernel bounds, unweighted estimates were studied in
\cite{DY2} using the sharp maximal function introduced in \cite{Ma1}, \cite{Ma2}. In that case, a simpler proof, that also yields the weighted
estimates, was obtained in \cite{CMP} using  the
pointwise estimate
$
\big|[b,L^{-\alpha/2}] f(x)\big|
\lesssim
I_\alpha (|b(x)-b|\,|f|)(x).
$
A discretization method inspired by \cite{Pe2}  is used to show that
the latter operator is controlled in $L^1(w)$ by $M_{L\,\log L,
\alpha} f$ for every $w\in A_\infty$. From
here, by the extrapolation techniques developed in
\cite{CMP}, this control can be extended to $L^p(w)$ for $0<p<\infty$,
$w\in A_\infty$ and consequently the weighted estimates of
$[b,L^{-\alpha/2}]$ reduce to those of $M_{L\,\log L, \alpha}$ which
are studied in \cite{CF}.
\end{remark}

\subsection{Riesz transform on manifolds}

Let $M$ be a complete non-compact  Riemannian manifold with $d$ its
geodesic distance and $\mu$ the volume form. Let  $\Delta$ be the positive Laplace-Beltrami
operator on $M$ given by
$$
\langle\Delta f, g\rangle = \int_M \nabla f \cdot \nabla g \, d\mu
$$
where $\nabla$ is the Riemannian gradient on $M$ and $\cdot$ is an inner product on $TM$. The Riesz transform is the tangent space valued operator   $\nabla \Delta^{-1/2}$ and it is bounded from $L^2(M,\mu)$ into $L^2(M;TM, \mu)$ by construction.

The manifold $M$ verifies the
doubling volume property if $\mu$ is doubling:
$$
\mu(B(x,2\,r))\le C\,\mu(B(x,r))<\infty, \eqno(D)
$$
for all $x\in M$ and $r>0$ where $B(x,r)=\{y\in M: d(x,y)<r\}$. A Riemannian manifold
$M$ equipped with the geodesic distance and a doubling volume form is a
space of homogeneous type. Non-compactness of $M$  implies  infinite
diameter, which together with the doubling volume property yields
$\mu(M)=\infty$ (see for instance
\cite{Ma2}).

One says that the heat kernel $p_t(x,y)$ of the semigroup $e^{-t\Delta}$ has Gaussian upper bounds if  for some constants $c,C>0$ and all $t>0, x,y \in M$,
$$
p_t(x,y) \le  \frac  C{\mu(B(x,\sqrt t))} \, e^{-c \frac{d^2(x,y)}{t}}. \eqno(GUB)
$$
It is known that under doubling  it is a consequence of the same inequality only at  $y=x$  \cite[Theorem 1.1]{G}.

\begin{theor}[\cite{CD}]\label{Rp>2}
Under $(D)$ and $(GUB)$, then
$$
\big\|\,|\nabla\Delta^{-1/2}f|\,\big\|_p \le C_p
\|f\|_p \eqno (R_{p})
$$
holds for $1<p< 2$ and all $f\in L^\infty_c(M)$.
\end{theor}

 Here,  $|\cdot|$ is the norm on $TM$ associated with the inner product.

We shall set
$$
q_{+}=\sup\big\{p\in (1,\infty)\, : \, (R_p) \ {\rm holds}\big\}.
$$
which  satisfies $q_{+}\ge 2$ under the assumptions of Theorem \ref{Rp>2}. It can
be equal to 2 (\cite{CD}). It is  bigger than 2 assuming further  the stronger $L^2$-Poincar\'{e} inequalities  (\cite{AC}). It can be  equal to $+\infty$ (see below).

Let us turn to weighted estimates.

\begin{theor}[\cite{AM4}]\label{theor:Rp:w}
Assume $(D)$ and $(GUB)$. Let $w\in A_{\infty}(\mu)$.
\begin{list}{$(\theenumi)$}{\usecounter{enumi}\leftmargin=.85cm
\labelwidth=.7cm\labelsep=0.25cm \itemsep=0.15cm \topsep=.3cm
\renewcommand{\theenumi}{\roman{enumi}}}

\item For $p\in \W_{w}(1, q_{+})$, the Riesz transform is of strong-type $(p,p)$ with respect to $w\,d\mu$, that is,
\begin{equation}\label{eq:Rp:w}
\big\|\, |\nabla \Delta^{-1/2}f|\, \big\|_{L^p(M,w)}
\le
C_{p,w}\, \|f\|_{L^p(M,w)}
\end{equation}
for all $f\in L^\infty_c(M)$.

\item If  $w\in A_{1}(\mu) \cap RH_{(q_{+})'}(\mu) $, then the
Riesz transform is of weak-type $(1,1)$ with respect to $w\,d\mu$,
that is,
\begin{equation}\label{eq:R1:w}
\big\|\, |\nabla \Delta^{-1/2}f|\, \big\|_{L^{1,\infty}(M,w)}
\le
C_{1,w}\, \|f\|_{L^1(M,w)}
\end{equation}
for all $f\in L^\infty_c(M)$.
\end{list}

\end{theor}

Here, the strategy of proof is a little bit different. Following ideas of \cite{BZ}, part $(i)$ uses the tools to prove Theorem \ref{theor:big}, namely a good-$\lambda$ inequality, together with a duality argument. For part $(ii)$, it uses a weighted variant of Theorem  \ref{theor:small}.
The operator $\A_{B}$ is given by $I- (I-e^{-r^2\, \Delta})^m$ with $m$ large enough and $r$ the radius of $B$. Note that here, the heat semigroup satisfies unweighted $L^1-L^\infty$ off-diagonal estimates on balls from $(GUB)$, so the kernel of $\A_{B}$ has a pointwise upper bound.

\begin{remark}\rm
Given $k\in\NN$ and $b\in\BMO(M,\mu)$ one can consider the $k$-th order
commutator of the Riesz transform  $(\nabla \Delta^{-1/2})_b^k$. This
operator satisfies \eqref{eq:Rp:w}, that is, $(\nabla
\Delta^{-1/2})_b^k$ is bounded on $L^p(M,w)$ under the same
conditions on $M,w,p$.
\end{remark}

If   $q_{+}=\infty$  then the Riesz transform is bounded on
$L^p(M,w)$ for $r_{w}<p<\infty$, that is, for $w\in A_{p}(\mu)$, and
we obtain the same weighted theory as for the Riesz transform on
$\RR^n$ :

\begin{corol}[\cite{AM4}]\label{corol:q+infty}
Let $M$ be a complete non-compact Riemannian manifold satisfying the
doubling volume property and Gaussian upper bounds. Assume that the Riesz
transform has strong type $(p,p)$ with respect to $d\mu$  for all $1<p<\infty$. Then the Riesz
transform has strong type $(p,p)$ with respect to $w\, d\mu$ for all $w\in A_{p}(\mu)$ and
$1<p<\infty$ and it is of weak-type $(1,1)$ with respect to $w\,
d\mu$ for all $w\in A_{1}(\mu)$.
\end{corol}

Unweighted $L^p$ bounds for Riesz transforms in different specific
situations were reobtained in a unified manner in \cite{ACDH} assuming conditions on the heat kernel and its gradient. The
methods used there are precisely those which allowed us to start the
weighted theory in \cite{AM1}.

Let us  recall
three situations in which this corollary applies  (see
\cite{ACDH}, where more is done,  and the references therein): manifolds with non-negative Ricci curvature, co-compact covering manifolds with
polynomial growth deck transformation group, Lie groups with polynomial volume growth endowed with a sublaplacian. A situation where $q_+<\infty$ is conical manifolds with compact basis without boundary.


\begin{thebibliography}{AHLMT}
\parskip=0.1cm




\bibitem[Aus]{Aus} P.~Auscher, {\em On necessary and sufficient conditions for $L^p$ estimates of Riesz transform associated elliptic operators on $\re^n$ and related estimates},  Mem. Amer. Math. Soc. \textbf{186} (871) (2007).

\bibitem[AC]{AC} P.~Auscher \& T.~Coulhon,  {\em Riesz transforms on manifolds and Poincar\'{e} inequalities}, Ann. Sc. Norm. Super. Pisa Cl. Sci. (5) \textbf{4}  (2005), 1--25.


\bibitem[ACDH]{ACDH} P.~Auscher, T.~Coulhon, X.T.~Duong \& S.~Hofmann, {\em Riesz transforms on manifolds and heat kernel regularity}, Ann. Scient. ENS Paris \textbf{37} (2004), no. 6, 911--957.


\bibitem[AHLMT]{AHLMcT} P.~Auscher, S.~Hofmann, M.~Lacey, A.~M$^{\rm c}$Intosh \& Ph. Tchamitchian, {\em The solution of the Kato square root problem for second order elliptic operators on $\re^n$}, Ann. Math. (2) \textbf{156} (2002), 633--654.

\bibitem[AM1]{AM1} P.~Auscher \& J.M.~Martell, {\em Weighted norm inequalities,  off-diagonal estimates and
elliptic operators.  Part I:  General operator theory and
weights}, Adv. Math.  \textbf{212}  (2007),  no. 1, 225--276.

\bibitem[AM2]{AM2} P.~Auscher \& J.M.~Martell, {\em Weighted norm inequalities,  off-diagonal estimates and
elliptic operators. Part II: Off-diagonal estimates on spaces of
homogeneous type},  J. Evol. Equ.  \textbf{7}  (2007),  no. 2, 265--316.

\bibitem[AM3]{AM3} P.~Auscher \& J.M.~Martell, {\em Weighted norm inequalities,  off-diagonal estimates and
elliptic operators.  Part III: Harmonic analysis of elliptic
operators},  J. Funct. Anal.  \textbf{241}  (2006),  no. 2, 703--746.

\bibitem[AM4]{AM4} P.~Auscher \& J.M.~Martell, {\em Weighted norm inequalities, off-diagonal estimates and elliptic operators.  Part IV: Riesz transforms on manifolds and weights}, Math. Z. \textbf{260} (2008), no. 3, 527--539.

\bibitem[AM5]{AM5} P.~Auscher \& J.M.~Martell, {\em Weighted norm inequalities for fractional  operators}, Indiana Univ. Math. J.  \textbf{57} (2008), no. 4,  1845--1870.


\bibitem[AT]{AT} P.~Auscher \& Ph.~Tchamitchian, {\em Square root problem for divergence operators and related topics}, Ast\'{e}risque Vol. 249, Soc. Math. France, 1998.

\bibitem[BZ]{BZ} F.~Bernicot \& J.~Zhao, {\em Abstract Hardy Spaces}, J. Funct. Anal.  \textbf{255}  (2008), no. 7, 1761--1796.


\bibitem[BK1]{BK1} S.~Blunck \& P.~Kunstmann, {\em Calder\'{o}n-Zygmund theory for non-integral operators and the $H^\infty$-functional calculus}, Rev. Mat. Iberoamericana  \textbf{19}  (2003),  no. 3, 919--942.

\bibitem[BK2]{BK2} S.~Blunck \& P.~Kunstmann, \emph{Weak-type $(p,p)$ estimates for Riesz transforms}, Math. Z.  \textbf{247}  (2004),  no. 1, 137--148.



\bibitem[CP]{CP}  L.A.~Caffarelli \& I.~Peral, {\em On $W^{1,p}$ estimates for elliptic equations in divergence form}, Comm. Pure App. Math. \textbf{51} (1998), 1--21.


\bibitem[Cha]{Chan} S.~Chanillo, {\em A note on commutators}, Indiana Univ. Math. J. \textbf{31} (1982), 7--16.

\bibitem[CD]{CD} T.~Coulhon \& X.T.~Duong,  {\em Riesz transforms for $1\le p\le 2$}, Trans. Amer. Math. Soc.  \textbf{351} (1999), 1151--1169.



\bibitem[CF]{CF} D.~Cruz-Uribe \& A.~Fiorenza, {\em Endpoint estimates and weighted norm inequalities for commutators of fractional integrals}, Publ. Mat. \textbf{47} (2003), 103--131.


\bibitem[CMP]{CMP} D.~Cruz-Uribe, J.M.~Martell \& C.~P\'{e}rez, {\em Extrapolation results for $A_\infty$ weights and applications}, J. Funct. Anal. \textbf{213} (2004), 412--439.

\bibitem[DY1]{DY1} X.T.~Duong \& L.~Yan, {\em Commutators of BMO functions and singular integral operators with non-smooth kernels},  Bull. Austral. Math. Soc.  \textbf{67}  (2003),  no. 2, 187--200.

\bibitem[DY2]{DY2} X.T.~Duong \& L.~Yan, {\em On commutators of fractional integrals},  Proc. Amer. Math. Soc.  \textbf{132}  (2004), no.~12, 3549--3557.


\bibitem[GR]{GR} J.~Garc\'{\i}a-Cuerva \& J.L.~Rubio de Francia, {\em Weighted Norm Inequalities and Related Topics}, North Holland Math.~Studies 116, North Holland, Amsterdam, 1985.

\bibitem[Gra]{Gra} L.~Grafakos, {\em Classical and Modern Fourier Analysis}, Pearson Education, New Jersey, 2004.

\bibitem[Gri]{G} A.~Grigor'yan, {\em Gaussian upper bounds for the heat kernel on arbitrary manifolds}, J. Differential Geom.  \textbf{45}  (1997),  no. 1, 33--52.


\bibitem[HM]{HM}   S.~Hofmann \& J.M.~Martell, {\em $L^p$ bounds for Riesz transforms and square roots associated to second order elliptic operators}, Pub. Mat. \textbf{47} (2003), 497--515.


\bibitem[H\"{o}r]{Hor} L.~H\"{o}rmander, {\em Estimates for translation invariant operators in $L^p$ spaces}, Acta Math. \textbf{104} (1960), 93--140.


\bibitem[KW]{KW} N.~Kalton \& L.~Weis, {\em the $H^\infty$-calculus and sums of closed operators},  Math. Ann. \textbf{321} (2001), 319--345.

\bibitem[LeM]{LeM} C.~Le Merdy, {\em On square functions associated to sectorial operators}, Bull. Soc. Math. France  \textbf{132}  (2004),  no. 1, 137--156.

\bibitem[Ma1]{Ma1} J.M.~Martell, {\em Sharp maximal functions associated with approximations of the identity in spaces of homogeneous type and applications}, Studia Math. \textbf{161} (2004), 113--145.

\bibitem[Ma2]{Ma2} J.M.~Martell, {\em Desigualdades con pesos en el An\'{a}lisis de Fourier: de los espacios de tipo homog\'{e}neo a las medidas no doblantes}, Ph.D. Thesis, Universidad Aut\'{o}noma de Madrid, 2001.

\bibitem[McI]{Mc} A.~M$^{\rm c}$Intosh, {\em Operators which have an $H^\infty$ functional calculus}, {Miniconference on operator theory and partial differential equations}, Vol. 14, Center for Math. and Appl., 210--231, Canberra, 1986. Australian National Univ.

\bibitem[Per]{Pe2} C.~P\'{e}rez, {\em Sharp $L^p$-weighted Sobolev inequalities}, Ann. Inst. Fourier (Grenoble)  45  (1995),  no.~3, 809--824.


\bibitem[Sh1]{Shen1} Z.~Shen, {\em The $L^p$ Dirichlet problem for elliptic systems on Lipschitz domains},
Math Res. Letters \textbf{13} (2006),  143--159.

\bibitem[Sh2]{Shen2} Z.~Shen,  {\em Bounds of Riesz transforms on $L^p$ spaces for second order elliptic operators},  Ann. Inst. Fourier \textbf{55} (2005), no. 1,  173--197.

\bibitem[Ste]{Ste} E.M.~Stein, {\em Singular integrals and differentiability of functions}, Princeton Univ. Press, 1970.

\bibitem[ST]{ST} J.O.~Str\"{o}mberg \& A. Torchinsky, {\em Weighted Hardy spaces}, Lecture Notes in Mathematics 1381, Springer-Verlag, 1989.


\bibitem[Wei]{W} L.~Weis, {\em Operator-valued Fourier multiplier theorems and maximal $L^p$-regularity},  Math. Ann. \textbf{319} (2001), 735--758.

\end{thebibliography}
\end{document}